\documentclass[preprint,10pt]{article}

\usepackage{amsmath,amsthm,amssymb}
\usepackage{mathptmx}
\usepackage{graphicx,graphics}
\usepackage{tikz}
\usepackage{hyperref}
\newtheorem{theorem}{Theorem}[section]
\newtheorem{corollary}[theorem]{Corollary}
\newtheorem{lemma}[theorem]{Lemma}
\newtheorem{prop}[theorem]{Proposition}
\theoremstyle{definition}
\newtheorem{definition}[theorem]{Definition}
\newtheorem{observation}[theorem]{Observation}

\newtheorem{problem}{Problem}[section]

\newtheorem{remark}[theorem]{Remark}
\usepackage{mathrsfs}
\usepackage{enumerate}
\usepackage{float}
\usepackage[english]{babel}

\usepackage{graphicx}
\usepackage{tikz}
\usetikzlibrary{shapes.geometric,arrows}
\usetikzlibrary{mindmap}
\usepackage{tikz}
\usetikzlibrary{automata}
\usetikzlibrary{arrows}
\usetikzlibrary{positioning,calc}
\usetikzlibrary{graphs}
\usetikzlibrary{graphs.standard}
\usetikzlibrary{arrows,decorations.markings}
\usepackage{tkz-graph}
\usetikzlibrary{chains,fit,shapes}
\usetikzlibrary{calc}

\tikzset{every loop/.style={min distance=10mm,looseness=10}}
\tikzset{every state/.style={minimum size=2mm}}
\newcount\nodecount
\tikzgraphsset{
	declare={subgraph N}%
	{
		[/utils/exec={\global\nodecount=0}]
		\foreach \nodetext in \tikzgraphV
		{  [/utils/exec={\global\advance\nodecount by1}, 
			parse/.expand once={\the\nodecount/\nodetext}] }
	},
	declare={subgraph C}%
	{
		[cycle, /utils/exec={\global\nodecount=0}]
		\foreach \nodetext in \tikzgraphV
		{  [/utils/exec={\global\advance\nodecount by1}, 
			parse/.expand once={\the\nodecount/\nodetext}] }
	},
	declare={subgraph W}%
	{
		[/utils/exec={\global\nodecount=0}]
		\foreach \nodetext in \tikzgraphV
		{  [/utils/exec={\global\advance\nodecount by1}, 
			parse/.expand once={\the\nodecount/\nodetext}] }
	}
}

\numberwithin{equation}{section}
\author{\hspace{1cm} Eshwar Srinivasan \ and Ramesh Hariharasubramanian \\
	{{\footnotesize s.eshwar@iitg.ac.in},\ {\footnotesize  ramesh\_h@iitg.ac.in}}\\{\footnotesize Department of Mathematics, Indian Institute of Technology Guwahati, Guwahati, Assam 781039, India}}
\begin{document}
	\title{On semi-transitive orientability of circulant graphs}
	\maketitle
	
	\begin{abstract}
		A graph $G = (V, E)$ is said to be \textit{word-representable} if a word $w$ can be formed using the letters of the alphabet $V$ such that for every pair of vertices $x$ and $y$, $xy \in E$ if and only if $x$ and $y$ alternate in $w$. A \textit{semi-transitive} orientation is an acyclic directed graph where for any directed path $v_0 \rightarrow v_1 \rightarrow \ldots \rightarrow v_m$, $m \ge 2$ either there is no arc between $v_0$ and $v_m$ or for all $1 \le i < j \le m$ there is an arc between $v_i$ and $v_j$. An undirected graph is semi-transitive if it admits a semi-transitive orientation. For given positive integers $n, a_1, a_2, \ldots, a_k$, we consider the undirected circulant graph with set of vertices $\{0, 1, 2, \ldots, n-1\}$ and the set of edges$\{ij ~ | ~ (i - j) \pmod n$ or $(j-i) \pmod n$ are in $\{a_1, a_2, \ldots, a_k\}\}$, where $ 0 < a_1 < a_2 < \ldots < a_k < (n+1)/2$. Recently, Kitaev and Pyatkin have shown that every $4$-regular circulant graph is semi-transitive. Further, they have posed an open problem regarding the semi-transitive orientability of circulant graphs for which the elements of the set $\{a_1, a_2, \ldots, a_k\}$ are consecutive positive integers. 
		
		In this paper, we solve the problem mentioned above. In addition, we show that under certain assumptions, some $k(\ge5)$-regular circulant graphs are semi-transitive, and some are not. Moreover, since a semi-transitive orientation is a characterisation of word-representability, we give some upper bound for the representation number of certain $k$-regular circulant graphs.
		
	\end{abstract}

	\maketitle
	\pagestyle{myheadings}
	
	\section{Introduction}
	
	A graph $G = (V, E)$ is said to be \textit{word-representable} if a word $w$ can be formed using the letters of the alphabet $V$ such that for every pair of vertices $x$ and $y$, $xy \in E$ if and only if $x$ and $y$ alternate in $w$.A \textit{semi-transitive} orientation is an acyclic directed graph where for any directed path $v_0 \rightarrow v_1 \rightarrow \ldots \rightarrow v_m$, $m \ge 2$ either there is no arc between $v_0$ and $v_m$ or for all $1 \le i < j \le m$ there is an arc between $v_i$ and $v_j$. An undirected graph is semi-transitive if it admits a semi-transitive orientation. For detailed reading, we refer the reader to \cite{kitaev2008representable, kitaev2017comprehensive, kitaev2015words, halldorsson2011alternation, kitaev2008word}.
	
		In this paper, we study the semi-transitive orientability of circulant graphs and answer an open problem (Problem 3) posed by Kitaev and Pyatkin in \cite{kitaev2020semit}. Moreover, since a semi-transitive orientation is a characterisation of word-representability, we give some upper bound for the representation number of certain $k$-regular circulant graphs. For any graph $G$, $|G|$ denotes the \textit{order} of the graph. For a word $w$, $|w|$ denotes the length of the word. If $i$ is any positive integer, then for brevity, we denote $i \pmod n$ by $(i)_{n}$. For any vertex $x$ of $G$, we denote the neighbours of the vertex $x$ as $N_G(x)$. All graphs considered are simple and undirected.
	\subsection{Circulant graph}
	
	A \textit{circulant graph} $C(n; R)$ for a set $R = \{a_1, a_2, \ldots, a_k\}$ is a graph with the vertex set $\{0, 1, \ldots, n-1\}$ and an edge set $\{ij ~ | ~ (i - j) \pmod n$ or $(j-i) \pmod n$ are in $\{a_1, a_2, \ldots, a_k\}\}$, where $ 0 < a_1 < a_2 < \ldots < a_k < (n+1)/2$. Let $n, r$ be positive integers, $n \ge 2$ and $r < n/2$. Then, $C(n; r)$ consists of a collection of disjoint cycles. If $d = gcd(n, r)$, then there are $d$ such disjoint cycles and each has length $n/d$. We say that each of these cycles has \textit{period} $r$, \textit{length} $n/d$, and \textit{rotation} $r/d$.
	
		\begin{theorem}[\cite{boesh84}, Proposition 1]\label{con}
		Circulant graph $C(n; R)$ for a set $R = \{a_1, a_2,\ldots, a_k\}$ is connected iff $gcd(n, a_1, a_2,$\ldots$, a_k) = 1$.
	\end{theorem}
	
	\begin{theorem}[\cite{kamal2013}, Theorem 10]\label{iso}
		For $n \in \mathbb{N}$, $P_2 ~\square~ C(2n+1; R) \cong C(2(2n+1); 2R \cup \{2n + 1\}) \cong C(2(2n+1); 2dR \cup \{2n + 1\})$, where $gcd(2(2n + 1), d) = 1$.
	\end{theorem}
	\subsection{Word-representability}
	
	Suppose that $w$ is a word over some alphabet, and $x$ and $y$ are two distinct letters in $w$. We say that $x$ and $y$ \textit{alternate} in $w$ if, after deleting all letters except the copies of $x$ and $y$ in $w$, we either obtain a word $xyxy\ldots$ (of odd or even length) or a word $yxyx\ldots$ (of odd or even length). Hence, by definition, if $w$ has a single occurrence of $x$ and a single occurrence of $y$, then $x$ and $y$ alternate in $w$.
	\begin{definition}[\cite{kitaev2015words}, Definition 3.0.5]\label{wr}
		A graph $G = (V, E)$ is said to be \textit{word-representable} if a word $w$ can be formed using the letters of the alphabet $V$ such that for every pair of vertices $x$ and $y$, $xy \in E$ if and only if $x$ and $y$ alternate in $w$. We say that $w$ \textit{represents} $G$, and $w$ is called a \textit{word-representant} of $G$. Also, it is essential that $w$ contains each letter of $V$ at least once.
	\end{definition}
	\begin{remark}[\cite{kitaev2017comprehensive}, Remak 1]\label{r1}
		The class of word-representable graphs is hereditary. That is, every induced subgraph of a word-representable graph is also word-representable.
	\end{remark}
	A word is called \textit{k-uniform} if each letter occurs exactly $k$ times in it. A graph $G$ is \textit{$k$-word-representable} if it can be represented by a $k$-uniform word. The least $k$ for which a word-representant of a graph $G$ is $k$-uniform is called the \textit{representation number} of the graph $G$, and it is denoted by $\mathcal{R}(G)$.
	\begin{theorem}[\cite{kitaev2008representable}, Theorem 7]
		A graph is word-representable if and only if it is $k$-word-representable for some $k$
	\end{theorem}
	\begin{theorem}[\cite{kitaev13}, Theorem 18]\label{pr}
		For $n \ge 4$, $\mathcal{R}(Pr_n) = 3$.
	\end{theorem}
	\begin{prop}[\cite{kitaev2015words}, Proposition 3.2.7]\label{uv}
		Let $w = uv$ be a $k$-uniform word representing a graph $G$, where $u$ and $v$ are two, possibly empty, words. Then, the word $w' = vu$ also represents $G$.
	\end{prop}
	\begin{definition}[\cite{kitaev2015words}, Definition 3.0.13]
		The reverse of the word $w = w_1w_2 \ldots w_n$ is the word \\$r(w) = w_n \ldots w_2 w_1$.
	\end{definition}
	\begin{prop}[\cite{kitaev2015words}, Proposition 3.0.14]\label{rw}
		If $w$ is a word-representant of a graph $G$, then $r(w)$ also represents the graph $G$.
	\end{prop}
	
	For a word $w$, suppose that $\pi(w)$ is the permutation obtained from $w$ after removing all but its leftmost occurrence of each letter $x$. We call $\pi(w)$ as the \textit{initial permutation} of $w$. Similarly, suppose that $\sigma(w)$ is the permutation obtained from $w$ after removing all but its rightmost occurrence of each letter $x$. We call $\sigma(w)$ as the \textit{final permutation} of $w$. Furthermore, a word $w$ restricted to certain letters $x_1, \ldots, x_m$ is denoted by $w|_{\{x_1, \ldots, x_m\}}$. Let $A(w)$ denote the set of letters present in the word $w$. For instance, if $w = 35423214$, then $\pi(w) = 35421$, $\sigma(w) = 53214$, $w|_{\{1,2\}} = 221$, and $A(w) = \{1, 2, 3, 4, 5\}$.
	
	\begin{observation}[\cite{kitaev2008representable}, Observation 4]\label{pw}
		Let $w$ be a word-representant of $G$. Then $\pi(w)w$ also represents $G$.
	\end{observation}
	\begin{observation}[\cite{kitaev2008representable}, Observation 3]\label{xx}
		Let $w = w_1xw_2xw_3$ be a word representing a graph $G$, where $w_1$, $w_2$ and $w_3$ are possibly empty words, and $w_2$ contains no $x$. Let $X$ be the set of all letters that appear only once in $w_2$. Then $N_G(x) \subseteq X$.
	\end{observation}
	\subsection{Known results about word-representability of circulant graphs}
	
	Word-representability of circulant graphs is studied by Kitaev and Pyatkin in \cite{kitaev2020semit}. They proved the following results regarding $4$-regular circulant graphs.
	
	\begin{theorem}[\cite{kitaev2020semit}, Theorem 7]\label{Circ} The circulant graph $C(13;1,5)$ is a $4$-chromatic $4$-regular semi-transitive graph of girth $4$. 
	\end{theorem}
	\begin{lemma}[\cite{kitaev2020semit}, Lemma 9]\label{Circs} A circulant graph $C(n;1,2)$ is semi-transitive for each $n\ge 6$. 
	\end{lemma}
	\begin{theorem}[\cite{kitaev2020semit}, Theorem 8]\label{Circ2} Each $4$-regular circulant graph is semi-transitive. 
	\end{theorem}
	
	Further, they have given an example of a non-semi-transitive circulant graph, i.e., $C(14;1,3,4,5)$. Hence, a circulant graph may not be semi-transitive. The following section presents results regarding semi-transitive and non-semi-transitive circulant graphs.
	
	\section{Semi-transitive orientability of circulant graphs}\label{sem}
	A \textit{circulant graph} $C(n; a_1, a_2, \ldots, a_k)$ is a graph with the vertex set $\{0, 1, \ldots, n-1\}$ and an edge set $\{ij ~ | ~ (i - j) \pmod n$ or $(j-i) \pmod n$ are in $\{a_1, a_2, \ldots, a_k\}\}$, where $ 0 < a_1 < a_2 < \ldots < a_k < (n+1)/2$.
	
	\begin{theorem}[\cite{HEUBERGER2003153}, Theorem 1]
		Let $G \cong C(n; a_1, a_2, \ldots, a_k)$ be a connected circulant graph. Then, $G$ is bipartite if and only if $a_1, a_2, \ldots, a_k$ are odd and $n$ is even.
	\end{theorem}
	Since bipartite graphs are transitive, we get the following result as a corollary.
	\begin{corollary}
		$C(2n; a_1, a_2, \ldots, a_k)$ is transitive if $a_i$ is odd for all $1 \le m \le k$.
	\end{corollary}
	
	In \cite{kitaev2020semit}, Kitaev and Pyatkin have posed a problem that asks whether  $C(n; t, t+1, \ldots, k)$ for some integers $k$ and $t$ satisfying $k - t > 1$ is semi-transitive or not. The following result answers the problem negatively.
	\begin{theorem}
		$C(n; t, t+1, \ldots, 2t)$ is not semi-transitive for $2 < \frac{n+1}{5} \le t \le \frac{n-1}{4}$
	\end{theorem}
	\begin{proof}
		By the definition of a circulant graph, two vertices $i$ and $j$ are adjacent if and only if $(i-j) \pmod n$ or $(j-i) \pmod n$ are in $\{t, t+1, \ldots, 2t\}$. Thus, $i$ and $j$ are adjacent if and only if either of the following cases holds for all $i > j$. \begin{enumerate}
			\item $t \le |i-j| \le 2t$
			\item $n - 2t \le |i-j| \le n - t$.
		\end{enumerate} 
		
		Consider an induced subgraph, $H$, with vertices $\{0, t-1, t, 2t-1, 2t+1, n-t\}$. We claim that $H$ is isomorphic to $W_5$. According to \textit{Remark }\ref{r1}, this implies that the graph $C(n; t, t+1, \ldots, 2t)$ is not semi-transitive.\\
		
		Consider the following table $M = \{m_{ij}\}$ for all $i, j \in V(H)$, where $m_{ij} = |i - j|$
		\begin{center}
			\begin{tabular}{|c|c|c|c|c|c|c|}
				\hline
				\textbf{i/j}& \textbf{0} & \textbf{t-1} & \textbf{t} & \textbf{2t - 1} & \textbf{2t +1} & \textbf{n-t}\\\hline
				\textbf{0} & 0& t - 1 & t & 2t - 1 & 2t + 1 & n - t \\\hline	
				\textbf{t-1} & t - 1 & 0 & 1& t & t + 2 & n - 2t + 1 \\\hline	
				\textbf{t} & t & 1& 0 & t - 1 & t + 1 & n - 2t\\\hline
				\textbf{2t-1} & 2t - 1 & t & t - 1 & 0 & 2 & n - 3t + 1\\\hline
				\textbf{2t +1} & 2t + 1 & t + 2 & t + 1 & 2 & 0 & n - 3t - 1 \\\hline
				\textbf{n-t }& n - t & n - 2t + 1 & n - 2t & n - 3t +1 & n - 3t - 1 & 0\\\hline
			\end{tabular}
		\end{center}
		Since $2 < \frac{n+1}{5} \le t \le \frac{n-1}{4}$, we have $t < n - 3t + 1 \le 2t$ and $t \le n - 3t - 1 < 2t$. Hence, the adjacency matrix, $\{h_{ij}\}$, of the induced subgraph $H$ can be written as follows: 
		$$\{h_{ij}\} =\begin{cases}
			1 & \text{if } t \le m_{ij} \le2t \text{ or } n - 2t \le m_{ij} \le n - t,\\
			0 & \text{otherwise}.
		\end{cases} = \begin{pmatrix}
			0 & 0 & 1 & 1 & 0 & 1\\
			0 & 0 & 0 & 1 & 1 & 1\\
			1 & 0 & 0  & 0 & 1 & 1\\
			1 & 1 & 0 & 0 & 0 & 1\\
			0 & 1 & 1 & 0 & 0 & 1\\
			1 & 1 & 1 & 1 & 1 & 0
		\end{pmatrix}$$
		Hence, the induced subgraph, $H$, formed by the above adjacency matrix, is isomorphic to $W_5$, as shown in \textit{Figure }\ref{hw5}.
	\end{proof}
	\tikzset{every picture/.style={line width=0.75pt}} 
	\begin{figure}[h]
		\centering
		\begin{tikzpicture}[x=0.75pt,y=0.75pt,yscale=-0.75,xscale=0.75]
			\filldraw (336.6,98.65) circle (0.1cm);
			\filldraw  (395.02,141.03) circle (0.1cm);
			\filldraw  (372.84,209.65) circle (0.1cm);
			\filldraw  (300.7,209.68) circle (0.1cm);
			\filldraw (278.3,141.07) circle (0.1cm);
			\filldraw  (336.69,160.02) circle (0.1cm);
			\draw   (336.6,98.65) -- (395.02,141.03) -- (372.84,209.65) -- (300.7,209.68) -- (278.3,141.07) -- cycle ;
			\draw    (336.6,98.65) -- (336.69,160.02) ;
			\draw    (395.02,141.03) -- (336.69,160.02) ;
			\draw    (372.84,209.65) -- (336.69,160.02) ;
			\draw    (300.7,209.68) -- (336.69,160.02) ;
			\draw    (278.3,141.07) -- (336.69,160.02) ;
			
			\draw (332,77) node [anchor=north west][inner sep=0.75pt]   [align=center] {0};
			\draw (400,132) node [anchor=north west][inner sep=0.75pt]   [align=center] {t};
			\draw (374,221) node [anchor=north west][inner sep=0.75pt]   [align=center] {2t+1};
			\draw (280,221) node [anchor=north west][inner sep=0.75pt]   [align=center] {t-1};
			\draw (240,132) node [anchor=north west][inner sep=0.75pt]   [align=center] {2t-1};
			\draw (323,175) node [anchor=north west][inner sep=0.75pt]   [align=center] {n-t};

		\end{tikzpicture}
		\caption{Induced subgraph $H \cong W_5$}
		\label{hw5}
	\end{figure}
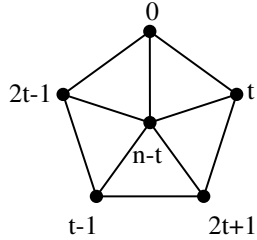
	
	The previous result shows that $C(n; t, t+1, \ldots, 2t)$ is not semi-transitive for $2 < \frac{n+1}{5} \le t \le \frac{n-1}{4}$. Hence, a natural question arises regarding semi-transitive orientability of $C(n; t, t+1, \ldots, 2t)$ for $t > \frac{n-1}{4}$. Interestingly, the next result positively answers this question for a more general case.
	\begin{theorem}\label{n4}
		$C(n; a_1, a_2, \ldots, a_k)$ is semi-transitive for all $a_1 \ge \frac{n+1}{4}$
	\end{theorem}
	\begin{proof}
		Consider a circulant graph $G = C(n; a_1, a_2, \ldots, a_k)$ with the vertex set $\{0, 1, \ldots, n-1\}$. Orient the edge $ij \in E(G)$ as $i \rightarrow j$ for all $i < j$, where $i, j \in V(G)$. We claim that the given orientation is semi-transitive. It is easy to see that the orientation is acyclic. Suppose that there is a shortcut $v_0 \rightarrow v_1 \rightarrow \ldots \rightarrow v_m$ with a shortcutting arc $v_0 \rightarrow v_m$ where $m \ge 3$. 
		
		Note that if $i \rightarrow j$, then $ j > i$ and $j \in \{i + a_1, \ldots, i + a_k, i + n - a_k, \ldots, i + n - a_1 \}$. Since $ a_1 \ge \frac{n +1}{4}$ and $a_k < \frac{n+1}{2}$, we have either $i + \frac{n+1}{4} \le j < i + \frac{n + 1}{2}$ or $ i + \frac{n - 1}{2} < j \le i + \frac{3n - 1}{4}$.
		
		Now since $v_l \rightarrow v_{l+1}$ for $l = \{0, 1, \ldots, m-1\}$, we have either $v_l + \frac{n+1}{4} \le v_{l+1} < v_l + \frac{n + 1}{2}$ or $ v_l + \frac{n - 1}{2} < v_{l+1} \le v_l + \frac{3n - 1}{4}$.
		
		From this we can see that $v_{l+1} \ge v_l + \frac{n+1}{4}$ for $l = \{0, 1, \ldots, m-1\}$. By solving this recursion, we get $v_{l} \ge v_0 + \frac{l(n+1)}{4}$. Since $m \ge 3$, we have $v_m \ge v_0 + \frac{3(n+1)}{4}$. 
		
		Since $v_0 \rightarrow v_m$, we have $v_m \le v_0 + \frac{3n - 1}{4} < v_0 + \frac{3(n+1)}{4} \le v_m$, which is a contradiction. Hence, the given orientation is semi-transitive.
	\end{proof}
	\begin{theorem}
		$C(n; t, t+1, \ldots, \lfloor \frac{n}{2} \rfloor)$ is semi-transitive $\forall ~ t = 1, 2, \ldots,  \lfloor \frac{n}{2} \rfloor$
	\end{theorem}
	\begin{proof}
		Consider a circulant graph $G = C(n; t, t+1, \ldots, \lfloor \frac{n}{2} \rfloor)$ with the vertex set $\{0, 1, \ldots, n-1\}$. Orient the edge $ij \in E(G)$ as $i \rightarrow j$ for all $i < j$, where $i, j \in V(G)$. We claim that the given orientation is semi-transitive. It is easy to see that the orientation is acyclic. Suppose that there is a shortcut $v_0 \rightarrow v_1 \rightarrow \ldots \rightarrow v_m$ with a shortcutting arc $v_0 \rightarrow v_m$ where $m \ge 3$. 
		
		Note that if $i \rightarrow j$, then $j > i$ and $j \in \{i + t, \ldots, i + n - t\}$. Now since $v_l \rightarrow v_{l+1}$ for $l = \{0, 1, \ldots, m-1\}$, we have $ v_l + t \le v_{l+1} \le v_l  + n - t$. Hence, by solving this recursion, we get $ v_0 + lt \le v_l \le v_0 + l(n - t)$.
		
		Consider $v_x$ and $v_y$, where $0 \le x < y \le m$. Since $v_0 \rightarrow v_m$, we have $v_0 + t \le v_m \le v_0 + n - t$. Hence, we have, $v_x + t < v_{x+1} \le v_y < v_m \le v_0 + n - t < v_x + n - t$. As a result, $v_x \rightarrow v_y$ for all $0 \le x < y \le m$. Hence, the vertices $\{v_0, v_1, \ldots, v_m\}$ induce a clique, which is a contradiction. Therefore, the given orientation is semi-transitive.
	\end{proof}
	\section{Representation number of some circulant graphs}\label{rep}
	
	In \cite{kitaev2020semit}, Kitaev and Pyatkin have posed a problem that asks whether  $C(n; 1, 2, \ldots, k)$ is semi-transitive or not. The following result answers the problem positively and gives its representation number.
	\begin{theorem}
		$C(n; 1, 2, \ldots, k)$ is $2$-word-representable.
	\end{theorem}
	\begin{proof}
		Consider a circulant graph $G = C(n; 1, 2, \ldots, k)$ with the vertex set $V(G) = \{0, 1, \ldots, n-1\}$. Define a morphism $h: V(G)^* \rightarrow V(G)^*$ as follows. $$h(i) = 
			i(i - k)_n$$ We claim that the word $w = h(u)$ represents $G$, where $u = 0~1~2~\ldots~(n-1)$. The word is of the form, $w = 0~(n-k)_n~1~(n-k+1)_n~2~\ldots~(n-1)_n~k~0~(k+1)_n~1 ~\ldots ~(n-1)_n~(n-k-1)_n$. By the definition of a circulant graph, $i$ is adjacent to $j$ if and only if $(j-i) \cong t \pmod n$, where $1 \le t \le k$. Let $w_i$ denote the subword of $w$ that occurs between the two occurrences of $i$.
			
			Suppose $ 0 \le i \le n - k -1$. Then, by the definition of morphism $h$, we have $w = \ldots~i~(i-k)_n~(i+1)_n~(i-k+1)_n~\ldots~(i-k-1)_n~(i-1)_n~(i+k)_n~i\ldots$. Note that $A(w_i) = \{(i-k)_n, (i-k+1)_n, \ldots, (i-1)_n, (i+1)_n, (i+2)_n, \ldots, (i+k)_n\} = N_G(i)$. Since $w$ is $2$-uniform, $i$ and $j \in N_G(i)$ alternate in $w$.
			
		Suppose $n - k \le i \le n-1$. Then, by the definition of morphism $h$, we have $w = \ldots 0~(n-k)_n~1~(n-k+1)_n~\ldots~(i+k)_n~i \ldots~i~(i-k)_n~(i+1)_n~(i-k+1)_n~ \ldots(n-2)_n~(n-k-2)_n~(n-1)_n~(n-k-1)_n$. Note that $A(w)\setminus A(w_i) = \{0, 1, 2, \ldots, (i+k)_n, (i-k)_n, (i-k+1)_n, \ldots. (n-1)_n, (n-k-1)_n\} = N_G(i)$. Since $w$ is $2$-uniform, $i$ and $j \in N_G(i)$ alternate in $w$. As a result, $w$ represents the graph $G$. 
	\end{proof}
	Every $2$-regular circulant graph, i.e., cycle graph, is $2$-word-representable. Is this true for $k$-regular graphs? 
	\begin{problem}\label{p1}
		Are word-representable $k$-regular circulant graphs $k$-word-representable?
	\end{problem}
	
	The following results answers \textit{Problem }\ref{p1} positively for $k = 3$.
	\begin{theorem}\label{3r}
		Let $G \cong C(2n; a, n)$ be a $3$-regular connected circulant graph with $gcd(a, 2n) = 1$. Then, $\mathcal{R}(G) \le 3$.  
	\end{theorem}
	\begin{proof}
		Consider a $3$-regular connected circulant graph $G \cong C(2n; a, n)$ with $gcd(a, 2n) = 1$. Let the vertex set of the graph be $V(G) = \{0, 1, \ldots, 2n - 1\}$. By the definition of a circulant graph, a vertex $i$ is adjacent to $j \in \{(i + a)_{2n}, (i - a)_{2n}, (i + n)_{2n}\}$. Define a morphism $h : V(G)^* \rightarrow V(G)^*$ as follows:
		\begin{equation*}
			h(i) = i(i-a)_{2n}(i+n)_{2n}
		\end{equation*}
		
		We claim that the word $w = h(u)$ represents $G$, where $u = 0~(a)_{2n}~(2a)_{2n}~(3a)_{2n}~\ldots ((2n - 1)a)_{2n}$. The word is of the form $w = 0~(2n-a)_{2n}~(n)_{2n}~(a)_{2n}~0~(a+n)_{2n}~\ldots(2n-a)_{2n}~(2n - 2a)_{2n}~(n-a)_{2n}$. Note that a vertex $i$ occurs once in every $h(i)$, $h((i+a)_{2n})$ and $h((i+n)_{2n})$. Therefore, each letter occurs three times in $w$. 
		
		Suppose $ i \in V(G) \setminus (2n - a)$. Then, by the definition of morphism $h$, we have $w = \ldots i~(i-a)_{2n}~(i + n)_{2n}~(i + a)_{2n}~i \ldots$. Therefore, $i$ and $j \in V(G)\setminus \{(i + a)_{2n}, (i - a)_{2n}, (i + n)_{2n}\}$ do not alternate in $w$. Now it is enough to prove that $i$ and $j \in \{(i + a)_{2n}, (i - a)_{2n}, (i + n)_{2n}\}$ alternate in $w$. If $i$ occurs to the left of $(i +n)_{2n}$ in $u$, we have $w = \ldots h(i)~h((i+a)_{2n})~h((i+2a)_{2n})~ \ldots ~h((i+n-a)_{2n})~h((i+n)_{2n})~h((i+n+a)_{2n})~\ldots$. Hence, $i$ and  $j \in \{(i + a)_{2n}, (i - a)_{2n}, (i + n)_{2n}\}$ alternate in $w$. If $i$ occurs to the right of $(i +n)_{2n}$ in $u$, we have $ w = \ldots ~h((i+n)_{2n})~h((i+n+a)_{2n})~\ldots ~h((i-a)_{2n})~h(i)~h((i+a)_{2n})~\ldots $.   Hence, $i$ and  $j \in \{(i + a)_{2n}, (i - a)_{2n}, (i + n)_{2n}\}$ alternate in $w$.
		
		Suppose $ i = 2n - a$. By the definition of a circulant graph, $2n - a$ is adjacent to $\{0, (2n - 2a)_{2n}, (n - a)_{2n}\}$. Then, by the definition of morphism $h$, we have $ w = h(0)~ \ldots ~h(n)~h((n-a)_{2n})~ \ldots h(2n-a)$. Hence, $2n - a$ and $j \in \{0, (2n - 2a)_{2n}, (n - a)_{2n}\}$ alternate in $w$. Further, $2n - a$ and $ j \in V(G) \setminus \{0, (2n - 2a)_{2n}, (n - a)_{2n}\}$ do not alternate in $w$. Therefore, $w$ represents the graph $G$. This implies $\mathcal{R}(G) \le 3$.
		
	\end{proof}
	
	A natural question arises whether $C(2n; a,n)$ with $gcd(a, 2n) = 1$, is $2$-word representable or not. Since $C(4; 1,2) \cong K_4$ is $1$-word representable, it is enough to study $C(2n; a,n)$ for $n > 2$.
	\begin{theorem}\label{mobnwr}
		Let $G \cong C(2n; a, n)$ be a $3$-regular connected circulant graph with $gcd(a, 2n) = 1$. Then, $G$ is not $2$-word-representable for all $n > 2$.
	\end{theorem}
	\begin{proof}
		Consider the graph $G \cong C(2n; a,n)$ with $gcd(a, 2n) = 1$. Let the vertex set $V(G) = \{0, 1, \ldots 2n - 1\}$. By the definition of a circulant graph, a vertex $i$ is adjacent to $j \in \{(i + a)_{2n}, (i-a)_{2n}, (n+i)_{2n}\}$. Note that the set $\{(i + a)_{2n}, (i-a)_{2n}, (n+i)_{2n}\}$ is an independent set because of the following reasons:
		\begin{itemize}
			\item $(i +a) - (i - a) \cong 2a \pmod {2n}$. Since $gcd(a, 2n) = 1$, $2a \not \in \{a, n\}$. Therefore, $i+a$ and $i -a$ are not adjacent in $G$.
			\item $(n + i) - (i - a) \cong n - a \pmod {2n}$. Since $gcd(a, 2n) = 1$, $n - a \not \in \{a, n\}$. Therefore, $i-a$ and $n + i$ are not adjacent in $G$.
			\item $(n+i) - (i+a) \cong n - a \pmod {2n}$.Since $gcd(a, 2n) = 1$, $n - a \not \in \{a, n\}$. Therefore, $i+a$ and $n + i$ are not adjacent in $G$.
		\end{itemize}
		
		Suppose that $G$ is $2$-word representable for $n > 2$. Let $w$ be a $2$-uniform word-representant of the graph $G$ with a letter $i$ such that no other letter occurs twice between the two copies of $i$. The word mentioned above always exists, as its absence implies that the graph is empty, which is a contradiction. By \textit{Observation }\ref{xx}, the letters between the copies of $i$ are $\{i + a, i-a, n+i\}$. Without loss of generality, consider $ i = 0$. Using the fact that the set $\{(2n - a), a, n\}$ is an independent set, by \textit{Proposition }\ref{rw} and \textit{Proposition }\ref{uv}, we only need to consider three cases. 
		\begin{itemize}
			\item $w$ is of the form $0 ~ (2n - a) ~a ~ n ~ 0 $\ldots$ n \ldots a \ldots (2n - a) \ldots$. Consider the vertex $n - a$. 
			\begin{itemize}
				\item $n - (n-a) \cong a\pmod {2n}$. Therefore, $n$ and $n - a$ are adjacent in $G$.
				\item $(2n - a) - (n - a) \cong n \pmod {2n}$. Therefore, $2n - a$ and $n - a$ are adjacent in $G$.
				\item $(n-a) - a \cong n - 2a\pmod {2n}$. Since $gcd(a, 2n) = 1$, $n - 2a \not \in \{a, n\}$. Therefore, $a$ and $n - a$ are not adjacent in $G$.
			\end{itemize} 
			If we introduce $n-a$ in the word $w$, we get either $$w = 0 ~ (2n - a) ~ a ~ n ~ 0 \ldots (n - a) \ldots n \ldots (n-a) \ldots a \ldots (2n-a)\ldots$$ where $n-a$ is not alternating with $2n-a$, which is a contradiction, or $$w = 0 ~ (2n - a)~ a ~ n ~ 0 \ldots (n - a) \ldots n \ldots a\ldots (n-a) \ldots (2n-a)\ldots$$ where $n-a$ is alternating with $a$, which is also a contradiction. 
			\item $w$ is of the form $0 ~ a ~ (2n -a) ~n ~ 0 \ldots n \ldots (2n-a) \ldots a \ldots$. Consider the vertex $(n+a)$.
			\begin{itemize}
				\item $(n+a) - n \cong a \pmod{2n}$. Therefore, $n$ and $n+a$ are adjacent in $G$.
				\item $(n + a) - a \cong n \pmod {2n}$. Therefore, $a$ and $n+a$ are adjacent in $G$.
				\item $(2n - a) - (n + a) \cong n -2a \pmod{2n}$. Since $gcd(a, 2n) = 1$, $n - 2a \not \in \{a, n\}$. Therefore, $2n - a$ and $n + a$ are not adjacent in $G$.
			\end{itemize}
			If we introduce $n+a$ in the word $w$, we get either $$0 ~ a ~ (2n -a) ~ n ~ 0 \ldots (n+a) \ldots n \ldots (n+a) \ldots (2n-a) \ldots a \ldots$$ where $n+a$ is not alternating with $a$, which is a contradiction, or $$0 ~ a ~ (2n -a) ~ n ~ 0 \ldots (n+a) \ldots n \ldots (2n-a) \ldots a \ldots (n+a) \ldots$$ where $n+a$ is alternating with $2n - a$, which is also a contradiction.
			\item $w$ is of the form $0 ~ (2n - a)~ n ~ a ~ 0 \ldots a \ldots n \ldots (2n - a) \ldots$. Let us introduce $(n+a)$ in $w$. We get, $$w = 0 ~ (2n - a)~ n ~ a ~ 0 \ldots (n+a) \ldots a \ldots n \ldots (n+a) \ldots (2n - a) \ldots$$ If we introduce $n-a$ in $w$, we get $$w = 0 ~ (2n - a)~ n ~ a ~ 0 \ldots (n+a) \ldots a \ldots(n-a) \ldots n \ldots (n+a) \ldots (2n - a) \ldots (n-a) \ldots$$
			Here, $n+a$ and $n-a$ are alternating. But $(n+a) - (n-a) \cong 2a \pmod{2n}$, which is a contradiction.
		\end{itemize}
		Therefore, $G$ is not $2$-word representable for $n > 2$.
	\end{proof}
	\begin{theorem}
		Let $G \cong C(2n; a, n)$ be a $3$-regular connected circulant graph with $gcd(a, 2n) \not = 1$. Then, $\mathcal{R}(G) = 3$.
	\end{theorem}
	\begin{proof}
		Consider a $3$-regular connected circulant graph $G \cong C(2n; a, n)$ with $gcd(a, 2n) \not = 1$. Since $G$ is connected, by \textit{Theorem }\ref{con}, we have $gcd(a, n) = 1$. Since $gcd(a, 2n) \not = 1$, $a$ is even, and $n$ is odd. Therefore, by \textit{Theorem }\ref{iso}, $C(2n; a, n) \cong P_2 \square C(n; a/2)$. Since $gcd(a, n) = 1$, we have $gcd(a/2, n) = 1$. Therefore, $C(2n; a, n) \cong P_2 \square C_n =Pr_n$ where $Pr_n$ is a prism graph on $n$ vertices. Therefore, by \textit{Theorem }\ref{pr}, $\mathcal{R}(G) = 3$. 
	\end{proof}
	\begin{corollary}\label{3rc}
		Let $G \cong C(2n; a, n)$ be a $3$-regular connected circulant graph. Then, $\mathcal{R}(G) = 3$.
	\end{corollary}
	M\"{o}bius ladder of order $n$ on $2n$ vertices, denoted by $M_n$, is a simple graph obtained by joining the antipodal vertices in $C_{2n}$. It can be seen that $M_n$ is isomorphic to the circulant graph $C(2n; 1, n)$. The following result is a direct consequence of \textit{Corollary }\ref{3rc}
	\begin{corollary}
		$\mathcal{R}(M_n) = 3$ for all $n > 2$.
	\end{corollary}
	\bibliographystyle{plain} 
	\bibliography{ref} 
\end{document}